\documentclass[journal]{IEEEtran}

\usepackage{amsmath}
\usepackage{bbm}
\usepackage{amsfonts}
\usepackage{amsthm}
\usepackage{pgf,tikz}
\usepackage{tikz-cd}
\usepackage{pgfplots}
\usepackage{amssymb}
\usepackage{geometry}
\usetikzlibrary{arrows}
\usepackage{xcolor}
\usepackage{mathtools}
\usepackage{algorithm}
\usepackage{textcomp}
\usepackage{algpseudocode}
\usepackage{enumitem}
\usepackage{float}
\usepackage{graphicx}
\usepackage{tkz-graph}
\usepackage{tabularx}

\newtheorem{definition}{Definition}
\newtheorem{assumption}{Assumption}
\newtheorem{theorem}{Theorem}
\newtheorem{lemma}{Lemma}

\ifCLASSINFOpdf
\else
\fi
\begin{document}
%
\title{On Decentralized Tracking with ADMM for Problems with Time-Varying Curvature}
%
%
%

\author{Marie~Maros,~\IEEEmembership{Student Member,~IEEE,}
        and~Joakim~Jald\'{e}n,~\IEEEmembership{Senior Member,~IEEE}%
        }
\maketitle


\begin{abstract}
We analyze the performance of the alternating direction method of multipliers (ADMM) to track, in a decentralized manner, a solution of a stochastic sequence of optimization problems parametrized by a discrete time Markov process. The main advantage of considering a stochastic model is that we allow the objective functions to occasionally lose strong convexity and/or Lipschitz continuity of their gradients. Due to the stochastic nature of our model, the tracking statement is given in a mean square deviation sense.  
\end{abstract}
\IEEEpeerreviewmaketitle
\section{Introduction}
We consider the alternating direction method of multipliers (ADMM) applied to distributed and dynamic optimization problems of the form
\begin{equation}
\label{eq:standard_optimization}
\begin{aligned}
& \underset{\mathbf{x} \in \mathbb{R}^{p}}{\text{min}} & \sum_{i=1}^{n_g} f_i(\mathbf{x},\boldsymbol{\theta}_k) \, .
\end{aligned}
\end{equation}
Each function $f_i$ is assumed convex in $\mathbf{x}$, and privately known to node $i$ in an $n_g$-node network. 
When the problem is static, i.e., when $\theta_k = \theta$, ADMM can be used as a basis for iterative distributed solution to \eqref{eq:standard_optimization} that only requires local knowledge of $f_i$ and communication between neighboring nodes. The distributed nature remains in the dynamic setting when \eqref{eq:standard_optimization} is parameterized by a sequence $\{ \boldsymbol{\theta}_k \}_{k \geq 0}$ that characterizes the change of the objective functions $f_i(\cdot,\boldsymbol{\theta}_k)$ over time $k$,
but the optimal objective value $\mathbf{x}^\star(\boldsymbol{\theta}_k)$ of \eqref{eq:standard_optimization} itself becomes time dependent through $\boldsymbol{\theta}_k.$ Problems of this form frequently arise in wireless sensor networks \cite{MAP} and  other engineering applications \cite{targtrack, admmtrack, powergrids}.

Iterative optimization methods that are allowed only one iteration per time $k$, have previously been studied in both centralized and distributed settings \cite{MAP, nconvex,predcorr,gradtrack,admmtrack}. An algorithm is considered to be capable of \emph{tracking} the optimal solution if the iterates remain in a neighborhood of the solution as time tends to infinity. Most first and second order algorithms enjoy linear or faster convergence rates in the static case, given that the objective function has nontrivial upper and lower bounds on the curvature. However, distributed implementations generally require bounds on the curvature of each $f_i.$ By imposing further restrictions such as that these bounds hold uniformly over $k$, and imposing uniform bounds on the maximum change between two consecutive optimal points, i.e. $\| \mathbf{x}^{\star}(\boldsymbol{\theta}_k) - \mathbf{x}^{\star}(\boldsymbol{\theta}_{k+1}) \| \leq b,$
\cite{predcorr,gradtrack,admmtrack} extended such static convergence results to obtain explicit tracking results in the dynamic setting. To the authors' knowledge, only \cite{nconvex} deals with  non-strongly convex objective functions in the dynamic setting. However, to deal with the loss of curvature, \cite{nconvex} added quadratic perturbations to compensate for the missing curvature. 

There are interesting cases where uniform bounds on the curvature become restrictive. Consider for example the distributed least squares problem 
\begin{equation}
\label{eq:quadratic}
\begin{aligned}
& \underset{\mathbf{x} \in \mathbb{R}^p}{\text{min}} & \sum_{i=1}^{n_g} \frac{1}{2} \|\mathbf{H}_i^{(k)}\mathbf{x} - \mathbf{y}_i^{(k)} \|_2^2,
\end{aligned}
\end{equation}
where $\theta_k = (\mathbf{H}_1^{(k)},\ldots,\mathbf{H}_{n_g}^{(k)}, \mathbf{y}_1^{(k)}, \ldots, \mathbf{y}_{n_g}^{(k)})$. If the sequence $\{ \theta_k \}_{k \geq 0}$ allows for the case where either of $\mathbf{H}_1^{(k)}, \ldots,\mathbf{H}_n^{(k)}$ become rank deficient, as it would if the $\mathbf{H}_i^{k}$ were drawn from, for example, a non-degenerate Gaussian autoregressive (AR) process, there does not exist any nontrivial bounds on curvature that hold uniformly. To further exacerbate the problem, in the Gaussian AR example, there are no uniform bounds on $\| \mathbf{x}^{\star}(\theta_k) - \mathbf{x}^{\star}(\theta_{k+1}) \|$. Nonetheless, if there are bounds that hold most of the time, one may still expect the algorithms to track in a stochastic sense.

We address these cases by explicitly modeling $\{ \boldsymbol{\theta}_k \}_{k \geq 0}$ as the realization of a stochastic process $\{\boldsymbol{\Theta}_k \}_{k \geq 0}$. Consequently, we pose the tracking statement in a mean square sense, rather than as deterministic bounds. Given the time varying nature of the curvature our algorithm of choice is ADMM. It is clear from \cite{nonergodic} that ADMM converges to the optimal solution(s) with no restrictions on the step-size, $\rho > 0,$ as opposed to other iterative methods. Algorithms such as gradient descent will require the step size to be bounded by a quantity related to the Lipschitz continuity constant of the gradient. As we will allow the Lipschitz continuity constant and the strong convexity constant to become arbitrarily bad we are interested in using a method that is robust to step-size selection.
 Further, ADMM's linear convergence, and thus, tracking ability under the assumptions stated in \cite{admmtrack} holds regardless of choice of step-size. Unlike the work in \cite{stochastic} we do not require deterministic bounds on the curvature of the objective functions nor on the gradients' norms. Further \cite{stochastic} only proposes a centralized solution.

 An important aspect to consider is the notion of sequence memory. In \cite{gradtrack,admmtrack,predcorr} this is implicitly accomplished through the bound on $\| \mathbf{x}^{\star}(\theta_k) - \mathbf{x}^{\star}(\boldsymbol{\theta}_{k+1}) \|$, and through similar bounds on the variation of gradients. Herein, we make this notion explicit by modeling $\{\boldsymbol{\Theta}_k \}_{k \geq 0}$ as a first order Markov sequence. The first order assumption is not particularly restrictive, as we allow the state space $\mathcal{S}$, where $\boldsymbol{\Theta}_k \in \mathcal{S}$, to be abstract.
 
The outline of the paper is as follows. Section \ref{sec:problemformulation} precisely formulates the problem and introduces the necessary notation to state the main result and the assumptions under which it holds. In Section \ref{section:main} we state the paper's main contribution in the form of a theorem (Theorem \ref{theorem:main_theorem}) that bounds the mean square deviation to an optimal point. In Section \ref{section:proof} the main statements required to prove Theorem \ref{theorem:main_theorem} are given. However, the their proof is not included in this paper due to space restrictions, and is left for the extended version of the paper which can be found on ArXiV. Finally, we present concluding remarks in Section \ref{section:conclusions}.

%
\IEEEpeerreviewmaketitle

\section{Problem Formulation \label{sec:problemformulation}}
Our goal is to establish tracking guarantees for the time-varying optimization problem
\begin{subequations}
\label{eq:problem}
\begin{align}
&\underset{\mathbf{x} \in \mathbb{R}^{n p},\mathbf{z}^{m p}}{\min} \quad f(\mathbf{x},\boldsymbol{\Theta}_k) = \sum_{i=1}^{n} f_i(\mathbf{x}_i,\boldsymbol{\Theta}_{k}) \\
& \text{s.t.} \quad \mathbf{Ax} + \mathbf{Bz} = 0,\label{eq:consensus}
\end{align}
\end{subequations}
where the objective function $f$ is parametrized by the stochastic process $\{\boldsymbol{\Theta}_k\}_{k \geq 0}$ and the variables $\mathbf{x}_i \in \mathbb{R}^{p}$ denote the local copies of the primal variable belonging to each node $i.$
Further, the matrices $\mathbf{A} \triangleq [\mathbf{A}_s,\mathbf{A}_d]$ where $\mathbf{A}_s$ and $\mathbf{A}_d$ denote the block arc source and block arc destination matrices as defined in \cite{qlinear} and $\mathbf{B} \triangleq [-\mathbf{I}_{m p},-\mathbf{I}_{m p}],$ where $m$ is the number of directed edges in the network. The vector $\mathbf{z}$ can be partitioned in $m$ $p-$length sub-vectors $\mathbf{z}_{ij}$ each of them assigned to an edge in the graph. Then, the constraint \eqref{eq:consensus} enforces that the variables $\mathbf{x}_i,$ and $\mathbf{x}_j$ attain the same value as $\mathbf{z}_{ij},$ consequently enforcing consensus in a undirected and connected graph. 

We are interested in ADMM's tracking ability in the mean square deviation sense given that the stochastic process $\{\boldsymbol{\Theta}_k\}_{k \geq 0}$  is a Markov Process defined on a general state-space. We will work under the assumption that we can perform a full ADMM iteration at each change of $\boldsymbol{\Theta}_k.$ Note that the results can be easily extended if one allows for $K$ iterations from $k$ to $k+1.$ As we are considering the same time scale for process and algorithm, we will index iterates and process outcomes with the same time index $k \geq 0.$

ADMM solves optimization problems with equality constraints by sequentially updating the primal variables $\mathbf{x}$ and $\mathbf{z}$ by alternatively minimizing the augmented Lagrangian
\begin{align}
&\mathcal{L}(\mathbf{x},\mathbf{z},\boldsymbol{\lambda},\boldsymbol{\theta}_k) \triangleq f(\mathbf{x},\boldsymbol{\theta}_k) + \\
& \quad \quad \boldsymbol{\lambda}^T(\mathbf{Ax} + \mathbf{Bz}) + \frac{\rho}{2}\|\mathbf{Ax}+\mathbf{Bz}\|^2. \nonumber
\end{align}
Here $\boldsymbol{\lambda}$ denotes the dual multiplier associated to the equality constraint \eqref{eq:consensus} and $\rho > 0$ ADMM's step-size. After the primal updates are completed, a dual step in the positive direction of the gradient is taken. For completeness, ADMM is summarized in Algorithm \ref{alg:ADMM}. 
\begin{figure}
\begin{algorithm}[H]
 \caption{ADMM}\label{alg:ADMM}
 \begin{algorithmic}[1]
  \State Initialize $ \mathbf{z}^{[0]}$ and $\boldsymbol{\lambda}^{[0]} $ as indicated by Assumption \ref{assumption:second}. Set $k=1$.
  \State Each node $i$ observes $f_i(\cdot,\theta_k)$
  \State $\mathbf{x}$ update:
  \begin{equation} \label{eq:xupdate}
  \mathbf{x}^{(k)} = \underset{\mathbf{x}}{\text{arg min}} \quad \mathcal{L}(\mathbf{x},\mathbf{z}^{(k-1)},\boldsymbol{\lambda}^{(k-1)},\theta_k)
	\end{equation}
	\State Agents exchange $\mathbf{x}_i^{(k)}$ with immediate neighbours.
  \State  $\mathbf{z}$ update:
  \begin{equation} \label{eq:zupdate}
  \mathbf{z}^{(k)} = \underset{\mathbf{z}}{\text{arg min}} \quad \mathcal{L}(\mathbf{x}^{(k)},\mathbf{z},\boldsymbol{\lambda}^{(k-1)}, \theta_k)
  \end{equation}
  \State $\boldsymbol{\lambda}$ update:
  \begin{equation} \label{eq:lupdate}
  \boldsymbol{\lambda}^{(k)} = \boldsymbol{\lambda^{(k-1)}} + \rho (\mathbf{A}\mathbf{x}^{(k)} + \mathbf{B}\mathbf{z}^{(k)})
  \end{equation} 
  \State Set $ k \leftarrow k + 1$, and return to 2.
 \end{algorithmic}
\end{algorithm}
\end{figure}
Steps 1-7 admit distributed implementations requiring only the exchange of the updates $\mathbf{x}_{i}^{(k)}$ of each node $i$ with their immediate neighbours.

The choice of the initial value of $\boldsymbol{\lambda}^{(0)}$ is not arbitrary. In particular, when the problem is parametrized by a stochastic variable, the second order moment of the minimizer need not be bounded. Therefore, a warm start is required to keep the initial second order quantities bounded. This notion is formalized in Assumption \ref{assumption:second}. With  problems of the structure of \eqref{eq:problem} and using the initialization scheme in Assumption \ref{assumption:second} we have that $\boldsymbol{\lambda}^{(k)} = [\boldsymbol{\alpha}^{(k)T},-\boldsymbol{\alpha}^{(k)T}]$ holds for all $k$ as well as at optimality. This symmetry will allow us to make statements in terms of $\boldsymbol{\alpha}^{(k)}$ and $\boldsymbol{\alpha}^{\star}(\boldsymbol{\Theta}_k)$ instead of $\boldsymbol{\lambda}^{(k)}$ and $\boldsymbol{\lambda}^{\star}(\boldsymbol{\Theta}_k),$ where $\boldsymbol{\lambda}^{\star}(\boldsymbol{\Theta}_k)$ denotes the optimal dual multiplier associated to the constraint \eqref{eq:consensus} with $\boldsymbol{\lambda}^{\star}(\boldsymbol{\Theta}_k) = [\boldsymbol{\alpha}^{\star T}(\boldsymbol{\Theta}_k),-\boldsymbol{\alpha}^{\star T}(\boldsymbol{\Theta}_k)].$ Note that there may be more than one optimal dual multiplier associated to \eqref{eq:consensus} at any given time, but given that the objective functions are differentiable, the  multiplier $\boldsymbol{\alpha}^{\star}(\boldsymbol{\Theta}_k)$ is unique. Whenever the objective function is not strongly convex the primal optimal points $\mathbf{x}^{\star}(\boldsymbol{\Theta}_k)$ and $\mathbf{z}^{\star}(\boldsymbol{\Theta}_k)$ will not necessarily be unique. We will argue in Appendix \ref{appendix:lemma1}, which can be found in the extended version of this paper on ArXiv, that this does not cause problems.  Finally, in the absence of strong convexity and or smoothness ADMM has been established to converge at rate $\mathcal{O}(\frac{1}{k})$ \cite{nonsmooth}. However, this convergence statement is made in terms of $\mathbf{z}$ and $\boldsymbol{\lambda}.$ Consequently our statements will be made in terms of $\mathbf{z}$ and $\boldsymbol{\lambda}$ as well.

We now introduce the assumptions under which our main statement holds.

\begin{assumption}
\label{assumption:geometrically_ergodic}
Let $\{\boldsymbol{\Theta}_k\}_{k \geq 0}$ be a time homogeneous Markov process evolving over a general state space $\mathcal{S}.$ Further, let $P^s(\boldsymbol{\theta},A)$ denote the $s-$step transition probability from $\boldsymbol{\theta}$ to $A \subseteq \mathcal{B}(\mathcal{S}).$ The process $\{\boldsymbol{\Theta}_k\}_{ k \geq 0}$ is $\phi-$irreducible and aperiodic. Also, there exists a small set $\mathcal{C},$ constants $b < \infty,$ $\beta > 0$ and a function $V \geq 1$ finite at some $\boldsymbol{\theta} \in \mathcal{S}$ satisfying
\begin{equation}
\Delta V(\boldsymbol{\theta}) \leq - \beta V(\boldsymbol{\theta}) + b \mathbbm{1}_{\mathcal{C}}(\boldsymbol{\theta}).
\end{equation} 
Further, the chain is initialized such that $\boldsymbol{\Theta}_0 \sim \pi$ where $\pi$ is the unique stationary distribution of $\{\boldsymbol{\Theta}_k\}_{k \geq 0}.$
\end{assumption}
For the same of completeness we formalize here the definition of small set. The remaining notions can be found in \cite{booktweedie}.
\begin{definition}[Small set]\label{definition:smallset} A set $C$ is $(\nu,s)-$small there exits $\beta >0,$ and integer $s \geq 1$ and a measure $\nu$ such that
\begin{equation}
P^s(x,B) \geq \beta \nu(B),\,\forall B \in \mathcal{B}(\mathcal{S}),\,\forall \mathbf{x} \in \mathcal{C}. \label{eq:majorization}
\end{equation}
\end{definition}
Assumption \ref{assumption:geometrically_ergodic} guarantees, among other things, that the process converges to its unique stationary distribution $\pi$ geometrically fast. Further, it provides guarantees regarding recurrence to the set $\mathcal{C}.$
\begin{assumption}
\label{assumption:wellconditioned}
There exists a small set $\mathcal{C}$ such that objective functions $f(\cdot,\boldsymbol{\theta})$ are at least $\mu_{\mathcal{C}}-$strongly convex and have at most $L_{\mathcal{C}}-$Lipschitz continuous gradients in $\mathbf{x}$ for all $\boldsymbol{\theta} \in \mathcal{C}.$ This implies that within $\mathcal{C}$ we have that the objective function's condition number
\begin{equation}
\kappa_{\mathcal{C}} \triangleq \frac{L_{\mathcal{C}}}{\mu_{\mathcal{C}}} < K < \infty.
\end{equation}
Further, the objective functions $f(\cdot,\boldsymbol{\theta})$ are convex and differentiable for all $\boldsymbol{\theta} \in \mathcal{S}.$ Further, let $\mu(\boldsymbol{\theta})$ denote the strong convexity constant associated to $f(\cdot,\boldsymbol{\theta})$ and $L(\boldsymbol{\theta})$ the Lipschitz continuity constant of its gradient.
\end{assumption}
\begin{assumption}
\label{assumption:network}
The nodes are connected via an undirected connected graph.
\end{assumption}
Assumption \ref{assumption:network} is standard in the context of distributed optimization and guarantees that the graph's Laplacian matrix has a single zero eigenvalue. We will denote it's largest eigenvalue $\Gamma_L$ and its second largest eigenvalue $\gamma_L > 0.$
\begin{assumption}[Bounded variations]
\label{assumption:bounded variations}
The fourth order moments of the primal dual variations are bounded quantities, i.e. there exits finite constants $B_{\boldsymbol{\lambda}}^4$ and $B_{\mathbf{x}}^4$ such that
\begin{align}
& \mathbb{E}[\|\mathbf{x}^{\star}(\boldsymbol{\Theta}_i)-\mathbf{x}^{\star}(\boldsymbol{\Theta}_{i-1})\|^4 ] \leq B_{\mathbf{x}}^4 \\
& \mathbb{E}[\|\boldsymbol{\lambda}^{\star}(\boldsymbol{\Theta}_i)-\boldsymbol{\lambda}^{\star}(\boldsymbol{\Theta}_{i-1})\|^4] \leq  B_{\boldsymbol{\lambda}}^4.
\end{align}
\end{assumption}
\begin{assumption}[Warm start]
\label{assumption:second}
The optimization problem parametrized at time $0,$ can be solved to a desired level of $\epsilon_0$ accuracy. To do this, set $\boldsymbol{\lambda}^{[-1]} = \mathbf{0}$ and $\mathbf{z}^{[-1]} = \mathbf{0}.$ Perform ADMM iterates until desired level of accuracy $\epsilon_0.$ Then, quantities that fulfil $\epsilon_0$ accuracy are denoted $\mathbf{z}(0)$ and $\boldsymbol{\lambda}(0)$. More precisely we require that
\begin{equation}
\|\mathbf{u}(0)-\mathbf{u}^{\star}(\boldsymbol{\Theta}_{0}^{\star})\|_{\mathbf{G}} \leq \epsilon_0
\end{equation}
where $\mathbf{u}(k) \triangleq [\mathbf{z}^{T}(k),\boldsymbol{\alpha}^{T}(k)]^T$ and $\mathbf{G} \triangleq \text{diag}(\rho \mathbf{I},\frac{1}{\rho}\mathbf{I}).$
\end{assumption}
The assumption above allows us to have an initial bound on the distance to the minimizer. This assumption is only required if for the particular problem-process pair the second order moment of the solution is unbounded. Further, note that by starting off with $\boldsymbol{\lambda}(-1)$ and $\boldsymbol{z}(-1)$ set to $\mathbf{0}$ $\boldsymbol{\lambda}(0)$ and $\boldsymbol{z}(0)$ can be found in a distributed manner.
\section{Main Result and proof \label{section:main}}
In this section we formalize the paper's main result followed by the main ingredients required to establish its veracity. Due to space restrictions, the proofs can be found in the extended version of the paper which can be found on ArXiV.
\begin{theorem}
\label{theorem:main_theorem}
Given the optimization problem \eqref{eq:consensus} and under Assumptions \ref{assumption:geometrically_ergodic}-\ref{assumption:second} ADMM with step-sizes $\rho > 0$ provides with distributed iterates such that
\begin{align}
&\underset{k \to \infty}{\lim} \, \sup \quad \mathbb{E}[\|\mathbf{u}(k) - \mathbf{u}^{\star}(\boldsymbol{\Theta}_k)\|^2] \leq \\
& \quad \frac{2C}{(1-\gamma^{1/2})^2}\sqrt{B_1(B_{\mathbf{x}},B_{\boldsymbol{\lambda}})},
\end{align}
where $C < \infty$ and $\gamma < 1$ are constants depending on the parameters and probability densities in \eqref{eq:majorization} and $B_1(B_{\mathbf{x}},B_{\boldsymbol{\lambda}})$ is a polynomial in $B_{\mathbf{x}}$ and $B_{\boldsymbol{\lambda}}$ which can be written as
\begin{align*}
&\left( \frac{\rho m_g}{n_g}\right)^2B_{\mathbf{x}}^4 +  4 \left(\frac{\rho m_g^{3/2}}{n_g^{3/2}\sqrt{2\gamma_L}}\right)(B_{\mathbf{x}}^3B_{\boldsymbol{\lambda}}) + \\
& 4 \left(\frac{\sqrt{\rho m_g}}{\sqrt{n_g}(2\rho \gamma_L)^{3/2}} \right)(B_{\boldsymbol{\lambda}}^3B_{\mathbf{x}}) + \\
& 6 \left( \frac{m_g}{n_g\gamma_L}\right)B_{\mathbf{x}}^2B_{\boldsymbol{\lambda}}^2 + \left(\frac{1}{4\rho^2\gamma_L^2}\right)B_{\boldsymbol{\lambda}}^4.
\end{align*}
\end{theorem}

\subsection{Proof of Theorem \ref{theorem:main_theorem} \label{section:proof}}

This section is devoted to establishing Theorem \ref{theorem:main_theorem}. For this we will use some supporting Lemmas. The Lemmas we introduce are proven in the appendices which can be found in the extended version of this paper.
\begin{lemma}
\label{lemma:detcontraction}
Under Assumptions \ref{assumption:wellconditioned} and \ref{assumption:network} it holds for $k \geq 1$
\begin{subequations}
\begin{align}
&\|\mathbf{u}(k) - \mathbf{u}^{\star}(\boldsymbol{\theta}_k)\|_{\mathbf{G}} \leq   \\
& \frac{\|\mathbf{u}(k) - \mathbf{u}^{\star}(\boldsymbol{\theta}_{k-1})\|_{\mathbf{G}}}{\sqrt{1+\delta(\boldsymbol{\theta}_k)}} + \frac{g(\boldsymbol{\theta}_k,\boldsymbol{\theta}_{k-1})}{\sqrt{1+\delta(\boldsymbol{\theta}_k)}}, \label{eq:detcontraction}
\end{align}
where
\begin{align}
& g(\boldsymbol{\theta}_k,\boldsymbol{\theta}_{k-1}) \triangleq \frac{\sqrt{\rho m_g}}{\sqrt{n_g}}\boldsymbol{\Delta}\mathbf{x}^{\star}(\boldsymbol{\theta}_k,\boldsymbol{\theta}_{k-1})\\
& \quad + \frac{1}{\sqrt{2\rho \gamma_L}}\boldsymbol{\Delta}\nabla_{\mathbf{x}}^{\star}f(\boldsymbol{\theta}_k,\boldsymbol{\theta}_{k-1})
\end{align}
\end{subequations}
\begin{subequations}
\begin{align}
&\boldsymbol{\Delta}\mathbf{x}^{\star}(\boldsymbol{\theta}_k,\boldsymbol{\theta}_{k-1}) \triangleq \|\mathbf{x}^{\star}(\boldsymbol{\theta}_k)-\mathbf{x}^{\star}(\boldsymbol{\theta}_{k-1})\| \\
&\boldsymbol{\Delta}\nabla_{\mathbf{x}}f^{\star}(\boldsymbol{\theta}_k,\boldsymbol{\theta}_{k-1}) \triangleq\\
&\,\, \|\nabla_{\mathbf{x}}f(\mathbf{x}^{\star}(\boldsymbol{\theta}_k),\boldsymbol{\theta}_k) - \nabla_{\mathbf{x}}f(\mathbf{x}^{\star}(\boldsymbol{\theta}_{k-1}),\boldsymbol{\theta}_{k-1})\|, \nonumber
\end{align}
\end{subequations}
\begin{equation}
\label{eq:delta}
\delta(\boldsymbol{\theta}) \triangleq \min \left\{ \frac{(\phi-1)\gamma_L}{\phi \Gamma_L}, \frac{2\rho \mu(\boldsymbol{\theta})\gamma_L}{\rho^2 \Gamma_L\gamma_L + \phi L(\boldsymbol{\theta})^2}\right\},
\end{equation}
whenever $\theta$ leads to a strongly convex 
with $\phi > 1$ an arbitrary constant.
\end{lemma}
Proof of this lemma can be found in Appendix \ref{appendix:lemma1}. The essential difference with the proof provided in \cite{tracking} is that we must take into account the possible lack of strong convexity and Lipschitz continuity of the gradients and non-uniqueness of primal optimal points.

Note that the statement in Lemma \ref{lemma:detcontraction} assumes a deterministic quantity $\boldsymbol{\theta}_k.$ Note that if the bound holds for all deterministic quantities, it will also hold as an expectation is being taken on each side. 
For notational convenience, let $\sqrt{q}(\boldsymbol{\theta}) \triangleq 1/\sqrt{1+\delta(\boldsymbol{\theta})}$ and $L(k,\boldsymbol{\theta}_k) \triangleq \|\mathbf{u}^{(k)} - \mathbf{u}^{\star}(\boldsymbol{\theta}_k)\|.$ This allows us to rephrase \eqref{eq:detcontraction} as
\begin{align}
& L(k,\boldsymbol{\theta}_k) \leq \sqrt{q(\boldsymbol{\theta}_k)}L(k-1,\boldsymbol{\theta}_{k-1}) \\& + \sqrt{q(\boldsymbol{\theta}_k)}g(\theta_k,\theta_{k-1}).
\end{align}

Recall that we are interested in convergence in a mean square error sense and we will therefore to take squares on both sides and then the expectation. If we were to do that, we would eventually have to deal with the expected value of the cross product of the two terms in the RHS of \eqref{eq:detcontraction}. In order to avoid this we will take squares on both sides of \eqref{eq:detcontraction} but use the Peter-Paul\footnote{$(a+b)^2 \leq (1+\upsilon)a^2 + (1+\frac{1}{\upsilon})b^2,\,\forall \upsilon >0.$} inequality on the RHS. Then we obtain for any $\upsilon > 0$
\begin{align}
\label{eq:peterpaul}
L^2(k,\boldsymbol{\theta}_k) \leq (1+\upsilon)q(\boldsymbol{\theta}_k)L^2(k-1,\boldsymbol{\theta}_{k-1}) \\
+ \left(1+\frac{1}{\upsilon}\right)q(\boldsymbol{\theta}_k)g^2(\boldsymbol{\theta}_k,\boldsymbol{\theta}_{k-1}). \nonumber
\end{align}
For any deterministic sequence $\{\boldsymbol{\theta}_j\}_{k \geq 0}$ we can apply \eqref{eq:peterpaul} recursively obtaining
\begin{align}
&L^2(k,\boldsymbol{\theta}_k) \leq (1+\upsilon)^k\left( \Pi_{i=1}^k q(\boldsymbol{\theta}_i)\right)L^2(0,\boldsymbol{\theta}_0) + \nonumber\\
&\quad \left(1+\frac{1}{\upsilon}\right)g^2(\boldsymbol{\theta}_k,\boldsymbol{\theta}_{k-1})q(\boldsymbol{\theta}_k). \label{eq:iterative}
\end{align}
This will hold for any $k$ and for any deterministic sequence $\{\boldsymbol{\theta}_k\}_{k \geq 0}.$ It therefore also holds in expectation. Hence, since for any $k$ the statement 
 \begin{align}
&\mathbb{E}[L^2(k,\boldsymbol{\Theta}_k)] \leq (1+\upsilon)^k \times \nonumber\\
& \mathbb{E}[\left( \Pi_{i=1}^k q(\boldsymbol{\Theta}_i)\right)L^2(0,\boldsymbol{\Theta}_0)]  \nonumber\\
&\quad + \left(1+\frac{1}{\upsilon}\right)\mathbb{E}[g^2(\boldsymbol{\Theta}_k,\boldsymbol{\Theta}_{k-1})q(\boldsymbol{\Theta}_k)].
\end{align}
\begin{subequations}
is true, we have that is must hold that
\label{eq:stochastic_bound}
\begin{align}
&\underset{k \to \infty}{\limsup} \quad \mathbb{E}[L^2(k,\boldsymbol{\Theta}_k)] \leq \\
& \underset{k \to \infty}{\limsup}\,\,\epsilon_0(1+\upsilon)^k \mathbb{E}\left[ \Pi_{i=1}^k q(\boldsymbol{\Theta}_i)\right] + \\
& \underset{k \to \infty}{\limsup}\,\, \left(1+\frac{1}{\upsilon}\right)\sum_{i=1}^k (1+\upsilon)^{k-i}\times \\
& \quad\mathbb{E}\left[ g^2(\boldsymbol{\Theta}_i,\boldsymbol{\Theta}_{i-1}) \Pi_{j=i}^k q(\boldsymbol{\Theta}_i)\right]. \label{eq:last}
\end{align}
\end{subequations}
where we have used that $L^2(0,\boldsymbol{\Theta}_0) \leq \epsilon_0.$

The term in \eqref{eq:last} complicates the analysis due to the correlation between $g^2(\boldsymbol{\Theta}_k,\boldsymbol{\Theta}_{k-1})$ and $\Pi_{j=1}^k q(\boldsymbol{\Theta}_{i}).$ In order to not have to explicitly compute or bound \eqref{eq:last} we will break the term apart. We then obtain the bound
\begin{align}
&\mathbb{E}[g^2(\boldsymbol{\Theta}_i,\boldsymbol{\Theta}_{i-1})\Pi_{j=i}^k q(\boldsymbol{\Theta}_j)] \leq \\ \label{eq:complication}
& \sqrt{B_1(B_{\mathbf{x}},B_{\boldsymbol{\lambda}})}\sqrt{\mathbb{E}[\Pi_{j=i}^k q(\boldsymbol{\Theta}_i)]}.  \nonumber
\end{align}  
The derivation of \eqref{eq:complication}
is cumbersome and provides limited insights. We therefore leave the derivation of the bound for the appendix.
We so far have established the following bound
\begin{subequations}
\label{eq:easy_stochasticbound}
\begin{align}
&\underset{k \to \infty}{\limsup} \quad \mathbb{E}[L^2(k,\boldsymbol{\Theta}_k)] \leq \\
& \underset{k \to \infty}{\limsup}\,\, \epsilon_0(1+\upsilon)^k \mathbb{E}\left[ \Pi_{i=1}^k q(\boldsymbol{\Theta}_i) \right] + \\
& \underset{k \to \infty}{\text{lim }}\text{sup} \left(1+\frac{1}{\upsilon}\right)\sum_{i=1}^k (1+\upsilon)^{k-i}\times \nonumber \\
& \quad \left( \sqrt{B_1(B_{\mathbf{x}},B_{\boldsymbol{\lambda}})}\sqrt{\mathbb{E}[\Pi_{j=i}^k q(\boldsymbol{\Theta}_i)]} + \nonumber \right)
\end{align}
\end{subequations}
Note that for us to be able to establish that the claim in Theorem \ref{theorem:main_theorem} is true we need to make some claim regarding the decay rate of the quantity
\begin{align}
\label{eq:averagecontraction}
\mathbb{E}\left[ \Pi_{j=i}^{k}q(\boldsymbol{\Theta}_j)\right]
\end{align}
When the quantity $\mathbb{E}[L^2(0,\boldsymbol{\Theta}_0)]$ is bounded without the need of performing a warm start similar statements to that in Theorem \ref{theorem:main_theorem} can be made. These are left for an extended version of the paper.
In particular we need both quantities in  \eqref{eq:averagecontraction} to decay exponentially fast. These statements are provided in the following theorem.
\begin{theorem}
Under Assumptions 
\label{theorem:contraction}
\begin{equation}
\mathbb{E}[\Pi_{i=j}^n q(\boldsymbol{\Theta}_i)] \leq C \gamma^{k-j}
\end{equation}
for some $C < \infty$  and $\gamma < 1.$ The quantities $C$ and $\gamma$ depend on $m,$ $\beta$ and $\nu$ in Definition \ref{definition:smallset}.  
\end{theorem}
We then have that 
\begin{subequations}
\begin{align}
&\underset{k \to \infty}{\text{lim }}\text{sup} \, \mathbb{E}[L^2(k,\boldsymbol{\Theta}_k)] \leq  \underset{k \to \infty}{\text{lim }}\text{sup}\, \epsilon_0((1+\upsilon)\gamma)^k \nonumber \\
& \underset{k \to \infty}{\text{lim }}\text{sup} \left(1+\frac{1}{\upsilon}\right)C\sum_{i=1}^k ((1+\upsilon)\gamma^{1/2})^{k-i}\times \nonumber \\
& \quad \left( \sqrt{B_1(B_{\mathbf{x}},B_{\boldsymbol{\lambda}})} \right)
\end{align}
\end{subequations}
Hence, it is sufficient to select $\nu < \frac{1-\gamma}{\gamma}$ which can always be done. Then, we obtain
\begin{subequations}
\begin{align}
&\underset{k \to \infty}{\lim}\,\text{sup} \quad \mathbb{E}[L^2(k,\boldsymbol{\Theta}_k)] \leq C\left(1 + \frac{1}{\nu} \right)\times \\
& \quad \frac{1+\nu}{1-(1+\nu)\gamma^{1/2}}\sqrt{B_1(B_{\mathbf{x}},B_{\boldsymbol{\lambda}})}
\end{align}
\end{subequations}
where the function of $\nu$ is minimized by selecting $\nu = \frac{1-\gamma^{1/2}}{1+\gamma^{1/2}}$ which yields the bound
\begin{align}
&\underset{k \to \infty}{\lim}\,\text{sup} \quad \mathbb{E}[L^2(k,\boldsymbol{\Theta}_k)] \leq \\
& \frac{2C}{(1-\gamma^{1/2})^2}\sqrt{B_1(B_{\mathbf{x}},B_{\boldsymbol{\lambda}})}.
\end{align}
\section{Numerical Experiments}
Consider the optimization problem \eqref{eq:quadratic} where the measurement matrices $\mathbf{H}_i^{(k)}$ and the measurement vectors $\mathbf{y}_i^{(k)}$ are modelled as an order 1 Auto-Regressive (AR) process
\begin{align}
\mathbf{H}_i^{(k)} = (1-\epsilon)\mathbf{H}_i^{k-1} + \epsilon \mathbf{V}
\end{align}
where $\mathbf{V}_{ij} \sim \mathcal{N}(0,1)$ and $\epsilon = 0.01.$ The 10 nodes are connected via a randomly generated undirected connected graph where each connection independently appears with probability 0.5. The step size $\rho = 10$ for the entire simulation which consists of 10000 tracks of length 1000. Further, the estimated fourth order deviations appearing in Assumption \ref{assumption:bounded variations} are estimated to be $\hat{B}_x^{4} = 7.2 \cdot 10^{-3}$ and $\hat{B}_l^{4} = 7.3 \cdot 10^{-4}.$ The estimated primal and dual second order residuals are displayed in Figures \ref{fig:primal} and \ref{fig:dual} respectively.
\begin{figure}
\begin{center}
\scalebox{0.53}{\input{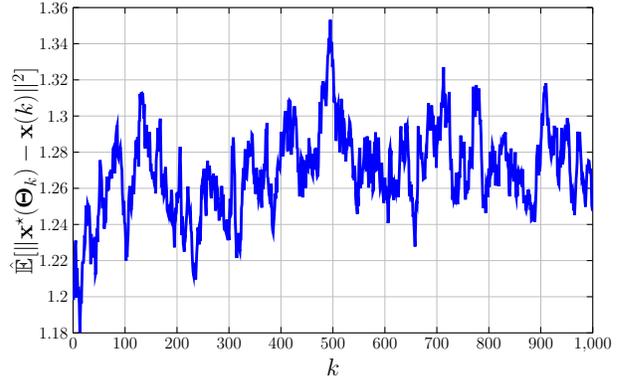}}
\end{center}
\caption{\label{fig:primal}Estimate of mean square deviation from the sequence of primal optimal points.}
\end{figure}
\begin{figure}
\begin{center}
\scalebox{0.53}{\input{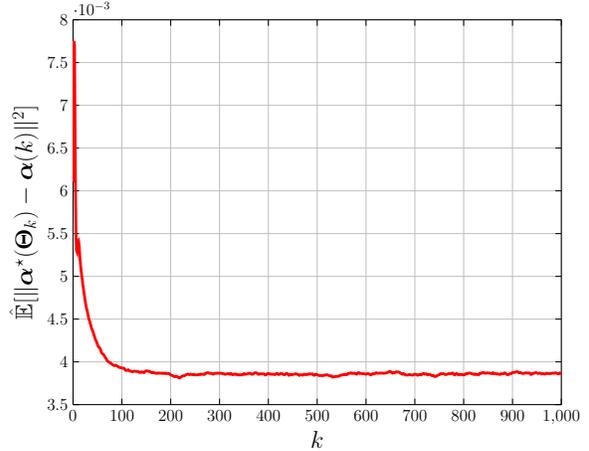}}
\end{center}
\caption{\label{fig:dual}Estimate of mean square deviation from the sequence of dual optimal points.}
\end{figure}

\section{Conclusion and Further work \label{section:conclusions}}
In this paper we establish that ADMM can track a solution in the mean square error sense even when smoothness and strong convexity are periodically lost. An interest line of further research is to extend the analysis for time-varying constraints so as to extend the applicability of the result.


%

\appendices

\section{Proof of Lemma \ref{lemma:detcontraction}\label{appendix:lemma1}}
Whenever the objective function is strongly convex and has Lipschitz continuous gradients with parameters $\mu(\boldsymbol{\theta})$ and $L(\boldsymbol{\theta})$ respectively the statement in Lemma \ref{lemma:detcontraction} readily follows from Theorem 1 in \cite{tracking}. The main additional issues we face with the statement in Lemma \ref{lemma:detcontraction} is showing the the same statement holds with $\delta(\boldsymbol{\theta}) = 0$ when the objective function is only differentiable and convex and deal with the lack of uniqueness in primal variables. 

For any primal optimizer $\mathbf{z}^{\star}(\boldsymbol{\Theta}_k)$ it holds true that
\begin{align}
& \|\mathbf{u}(k) - \mathbf{u}^{\star}(\boldsymbol{\Theta}_k)\|_{\mathbf{G}} \leq 
\|\mathbf{u}(k) - \mathbf{u}^{\star}(\boldsymbol{\Theta}_{k-1})\|_{\mathbf{G}} + \nonumber\\
&  \rho\|\mathbf{z}^{\star}(\boldsymbol{\Theta}_k)-\mathbf{z}^{\star}(\boldsymbol{\Theta}_{k-1})\| + \frac{1}{\rho}\|\boldsymbol{\alpha}^{\star}(\boldsymbol{\Theta}_k) - \boldsymbol{\alpha}^{\star}(\boldsymbol{\Theta}_{k-1})\|.
\end{align}
Note that this will hold regardless of the choice of $\mathbf{z}^{\star}(\boldsymbol{\Theta}_j)$ for both $j = k, k-1$ as long as the same choice is made within the same equation. Then, as in Theorem 1 in \cite{tracking} we have that
\begin{align}
\|\mathbf{z}^{\star}(\boldsymbol{\Theta}_{k}) - \mathbf{z}^{\star}(\boldsymbol{\Theta}_{k-1})\| =\\
 \frac{\sqrt{\rho m}}{\sqrt{n}}\|\mathbf{x}^{\star}(\boldsymbol{\Theta}_{k}) - \mathbf{x}^{\star}(\boldsymbol{\Theta}_{k-1})\|,
\end{align}
where again, the choice must be consistent in the sense that $\mathbf{z}^{\star}_{ij}(\boldsymbol{\Theta}_k) = \mathbf{x}^{\star}_i(\boldsymbol{\Theta}_k) = \mathbf{x}^{\star}_j(\boldsymbol{\Theta}_k)$ and must be consistent with the minimizers in Assumption \ref{assumption:bounded variations}.

When bounding the deviation in the multipliers we use the following optimality conditions
\begin{align}
\nabla_{\mathbf{x}} f(\mathbf{x}^{\star}(\boldsymbol{\Theta}_k),\boldsymbol{\Theta}_k) + \mathbf{E}_{\text{o}}^T\boldsymbol{\alpha}^{\star}(\boldsymbol{\Theta}_k) = \mathbf{0} \nonumber\\
\nabla_{\mathbf{x}} f(\mathbf{x}^{\star}(\boldsymbol{\Theta}_{k-1}),\boldsymbol{\Theta}_{k-1}) + \mathbf{E}_{\text{o}}^T\boldsymbol{\alpha}^{\star}(\boldsymbol{\Theta}_{k-1}) = \mathbf{0},
\end{align}
where the gradients take the same value regardless of the choice of primal minimizer. From here the proof follows analogously to that in Theorem 1 in \cite{tracking}.

\section{Bounding \eqref{eq:complication} \label{appendix:cauchy}}
In order to bound \eqref{eq:complication} we use the Cauchy-Schwarz inequality to state
\begin{equation}
\label{eq:rvbound}
\mathbb{E}[XY] \leq \sqrt{\mathbb{E}[X^2]\mathbb{E}[Y^2]}
\end{equation}
to obtain
\begin{align}
&\mathbb{E}[g^2(\boldsymbol{\Theta}_i,\boldsymbol{\Theta}_{i-1})\Pi_{j=i}^k q(\boldsymbol{\Theta}_j)] \leq \\
&\quad\sqrt{\mathbb{E}[g^4(\boldsymbol{\Theta}_i,\boldsymbol{\Theta}_{i-1})]\mathbb{E}[\Pi_{j=i}^k q^2(\boldsymbol{\Theta}_j)]}
\end{align}
We now in turn must bound the quantity
\begin{equation}
\mathbb{E}[g^4(\boldsymbol{\Theta}_i,\boldsymbol{\Theta}_{i-1})]
\end{equation}
For notational simplicity let $A \triangleq \boldsymbol{\Delta}\mathbf{x}^{\star}(\boldsymbol{\Theta}_i,\boldsymbol{\Theta}_{i-1}) $ and $B \triangleq \boldsymbol{\Delta}\nabla_{\mathbf{x}}f^{\star}(\boldsymbol{\Theta}_i,\boldsymbol{\Theta}_{i-1}).$
Then,
\begin{align}
&\mathbb{E}[g^4(\boldsymbol{\Theta}_i,\boldsymbol{\Theta}_{i-1})] = \\
& \left(\frac{\rho m}{n}\right)^2 \mathbb{E}[A^4] + 4 \left( \frac{(\rho m)^{3/2} }{n^{3/2}\sqrt{2\rho \gamma_L}}\right)\mathbb[A^3B]  + \\
& + 4\left(\frac{\sqrt{\rho m}}{\sqrt{n}(2\rho \gamma_L)^{3/2}} \right)\mathbb{E}[B^3A] \\
& 6 \left( \frac{m}{2n\gamma_L}\right)\mathbb{E}[A^2B^2] + \left(\frac{1}{4\rho^2\gamma_L^2}\right)\mathbb{E}[B^4].
\end{align}
We are now to find bounds on the polynomials in $A$ and $B.$ Note that we require to compute this bounds in a cumbersome manner as we have not assumed anything regarding their cross-correlation, but only on the boundedness of their fourth order moments. In turn, the bound on the fourth order moment, will imply a bound on the second order moment.  In particular by applying the bound \eqref{eq:rvbound} we obtain the following bounds
\begin{align}
& \mathbb{E}[B^3A] \leq \sqrt{\mathbb{E}[B^4]\mathbb{E}[B^2A^2]} \\
& \mathbb{E}[A^3B] \leq \sqrt{\mathbb{E}[A^4]\mathbb{E}[A^2B^2]} \\
&\mathbb{E}[A^2B^2] \leq \sqrt{\mathbb{E}[A^4]\mathbb{E}[B^4]}\\
&\mathbb{E}[AB] \leq \sqrt{\mathbb{E}[A^2]\mathbb{E}[B^2]}, 
\end{align} 
which are further bounded as
\begin{align}
&\mathbb{E}[B^3A] \leq \sqrt{\mathbb{E}[B^4]\left(\sqrt{\mathbb{E}[A^4]\mathbb{E}[B^4]} \right)} \\
& \mathbb{E}[A^3B] \leq \sqrt{\mathbb{E}[A^4]\left( \sqrt{\mathbb{E}[A^4]\mathbb{E}[B^4]}\right)}
\end{align}
Then we have that
\begin{align}
& \mathbb{E}[B^3A] \leq  B_{\boldsymbol{\lambda}}^3B_{\mathbf{x}} \\
& \mathbb{E}[A^3B] \leq  B_{\mathbf{x}}^3B_{\boldsymbol{\lambda}}\\
& \mathbb{E}[A^2B^2] \leq B_{\boldsymbol{\lambda}}^2B_{\mathbf{x}}^2 \\
& \mathbb{E}[AB] \leq B_{\boldsymbol{\lambda}}B_{\mathbf{x}}.
\end{align}
Now we have that
\begin{align}
\label{eq:above1}
&\mathbb{E}[g^4(\boldsymbol{\Theta}_i,\boldsymbol{\Theta}_{i-1})] \leq \left(\frac{\rho m}{n}\right)^2 B_{\boldsymbol{x}}^4 + \\
& 4\left( \frac{\rho m^{3/2}}{n^{3/2}\sqrt{2\gamma_L}}\right)\left(B_{\mathbf{x}}^3 B_{\boldsymbol{\lambda}}\right) +\nonumber \\
& 4 \left(\frac{\sqrt{\rho m}}{\sqrt{n}(2\rho \gamma_L)^{3/2}} \right)\left(B_{\boldsymbol{\lambda}}^3 B_{\mathbf{x}}\right) +\nonumber \\
& 6 \left(\frac{m}{n \gamma_L}\right)B_{\mathbf{x}}^2B_{\boldsymbol{\lambda}}^2 + \left(\frac{1}{4\rho^2 \gamma_L^2}\right)B_{\boldsymbol{\lambda}}^4\nonumber.
\end{align}
Let $B_1(B_{\mathbf{x}},B_{\boldsymbol{\lambda}})$ denote the RHS of \eqref{eq:above1}. 
Analogously we have 
\begin{align}
& \mathbb{E}_{\mathcal{C}}[g^2(\boldsymbol{\Theta}_i,\boldsymbol{\Theta}_{i-1})] \leq \left(\frac{\rho m}{n}\right)B_{\mathbf{x}}^2 + \label{eq:above2}\\
&2 \left(\sqrt{\frac{m}{n \gamma_L}}\right)B_{\mathbf{x}}B_{\boldsymbol{\lambda}} + \left(\frac{1}{2\rho \gamma_L}\right)B_{\boldsymbol{\lambda}}^2.\nonumber
\end{align}
Analogously let $B_2(B_{\mathbf{x}},B_{\boldsymbol{\lambda}})$ denote the RHS of \eqref{eq:above2}.
\section{Proof of Theorem \ref{theorem:contraction}\label{appendix:theorem}}
Our goal is to ultimately provide an exponential bound for the quantity $\mathbb{E}[\Pi_{i=j}^k q(\boldsymbol{\Theta}_i)].$ To do this we will have to use the geometric ergodicity properties of the Markov process to make claims regarding how much time it takes to revisit the set $\mathcal{C}.$ This in turn, will allow us to establish how many visits to the set $\mathcal{C}$ the process will make in a certain amount of time. For the bound on the expectation $\mathbb{E}[\Pi_{i=j}^k q(\boldsymbol{\Theta}_i)]$ to behave exponentially a certain independence between different draws of the Markov process is desirable. However, due to the non-atomic nature of the individual sets of the Markov process this is not possible in general state spaces. For this purpose we will rely on the split chain (see \cite{booktweedie}) for which we are able to construct an atomic set $\mathcal{C}\times\{1\}.$ For notational convenience let us define $\mathcal{C}_1 \triangleq \mathcal{C} \times \{1\}.$ 

From now on we will be working with the split chain. Note that it follows from Assumption \ref{assumption:geometrically_ergodic} that the split chain $\{\boldsymbol{\Phi}_n\}_{ n \geq 0}$ has an accessible atom $\mathcal{C}_1.$ By construction, whenever we reach $\mathcal{C} = \mathcal{C}_0 \cup \mathcal{C}_1$ we may understand that a coin is tossed, such that with probability $1-\beta$ we land on $\mathcal{C}_0$ and with probability $\beta$ we land on $\mathcal{C}_1.$ Further, the transition probability from $\mathcal{C}_1$ to any other point is independent on where from $\mathcal{C}_1$ we came from.
 Let $S_k$ denote the time index at which the set $\mathcal{C}_1$ is re-visited for the $k^{\text{th}}$ time. Then, let $\tau_k \triangleq S_k - S_{k-1},$ for $k \geq 1.$ The inter-renewal times $\tau_k$ can be shown to be i.i.d random variables \cite{kalashnikov}. For convenience, when $k$ is not relevant it will be dropped and we will talk about $\tau.$
 Further, let $N_{\mathcal{C}_1}(n-j)$ denote the number of visits to the set $\mathcal{C}_1.$ The process $\{\boldsymbol{\Phi}_k\}_{k \geq 0}$ makes to $\mathcal{C}_1$  in a time interval of length $n-j.$ 
\begin{lemma}
Let $\tau_h^i \triangleq \min_{k:S_k > i} S_k-i$ be the first time the process reaches $\mathcal{C}_1$ after time $i.$ Then, $\exists \, M_2 < \infty,$ and $\eta < 1$ such that $P(\tau_h^i =n ) \leq  M_2 \eta^n.$
\end{lemma}
\begin{proof} 
Let $\tau_s$ denote the time between visits to $\mathcal{C}$ corresponding to the Markov process $\{\boldsymbol{\Theta}_k\}_{k \geq 0}.$
It follows from Assumption \ref{assumption:geometrically_ergodic} that there exists a $\kappa > 1$such that
\begin{equation}
\underset{\boldsymbol{\theta} \in \mathcal{C}}{\text{sup}} \quad \mathbb{E}[\kappa^{\tau_s}|\boldsymbol{\theta}] < \infty.
\end{equation}
 The equation above implies that given that we start at any point $\theta \in \mathcal{C}$ the  return time to $\mathcal{C}$ follows a distribution with tails that decay at least exponentially.
From Assumption \ref{assumption:geometrically_ergodic} and \eqref{eq:majorization} it follows that for any $\boldsymbol{\theta}\in \mathcal{C}$ we have that
\begin{align}
&\mathbb{E}[\kappa^{\tau_{\mathcal{C}}}|\boldsymbol{\theta}] = \nonumber\\ 
&(1-\beta)\mathbb{E}[\kappa^{\tau_{\mathcal{C}_0}}|\boldsymbol{\theta}\times\{0\}] + \beta \mathbb{E}[\kappa^{\tau}|\boldsymbol{\theta}\times\{1\}]
\end{align}
and consequently
\begin{equation}
\underset{\boldsymbol{\phi} \in \mathcal{C}_1}{\text{sup}} \quad \mathbb{E}[\kappa^{\tau}|\boldsymbol{\Phi}] \leq \frac{1}{\beta}\underset{\boldsymbol{\theta} \in \mathcal{C}}{\text{sup}} \quad \mathbb{E}[\kappa^{\tau}_{\mathcal{C}}|\boldsymbol{\theta}] < \infty.
\end{equation}
 From this, it follows that $P(\tau = k) \leq M_1\left(\frac{1}{\kappa}\right)^{k}$ for some $M_1 < \infty. $ For convenience, let $\eta \triangleq \frac{1}{\kappa}.$ 
Before going into more details we will take a closer look into how to relate the quantities $\tau$ and $\tau_h^i.$
Note that we have the following equalities
\begin{align}
& P(\tau = n) = P(\boldsymbol{\Phi}_{i+n} \in \mathcal{C}_1,\boldsymbol{\Phi}_{i + n - 1} \not \in \mathcal{C}_1,\\
& \hdots,\boldsymbol{\Phi}_{i+1}\not \in \mathcal{C}_1 | \boldsymbol{\Phi}_{i} \in \mathcal{C}_1) \, \forall i \geq 0 \nonumber\\
& P(\tau_h^i = n ) = P(\boldsymbol{\Phi}_{i+1}\in \mathcal{C}_1,\boldsymbol{\Phi}_{i+n-1}\not \in \mathcal{C}_1, \\
& \hdots, \boldsymbol{\Phi}_{i+1} \not \in \mathcal{C}_1) \nonumber.
\end{align}
This will allow us to relate the two quantities later on.
Further, we have that for any $j \geq 0$ 
\begin{equation}
P(\tau_i^h = n | \boldsymbol{\Theta}_{i} \not \in \mathcal{C}_1,\hdots,\boldsymbol{\Theta}_{i-j} \in \mathcal{C}_1)
\end{equation}
does not depend on conditioning further than to the time instant in which the process visits $\mathcal{C}_1$ as this constitutes a renewal of the chain. Hence, we have that
\begin{align}
&P(\tau_h^{i}=n) = P(\tau_h^{i}=n,\boldsymbol{\Phi}_{i}\in\mathcal{C}_1) +\\
& \sum_{j=1}^{\infty}P(\tau_h^i = n,\boldsymbol{\Phi}_{i} \not \in \mathcal{C}_1,\hdots,\boldsymbol{\Phi}_{i-j}\in \mathcal{C}_1),\nonumber
\end{align}
where we may go arbitrarily back as the process is initialized in the stationary distribution and is therefore equivalent in behaviour as a process that has been always running. Further, we have that 
$P(\tau_h^{i} = n,\boldsymbol{\Phi}_i \not \in \mathcal{C}_1,\hdots, \boldsymbol{\Phi}_{i-j} \in \mathcal{C}_1) = \frac{P(\tau=n+j)}{\beta\pi(\mathcal{C})},$
implying
\begin{align}
\label{eq:taubounds}
&P(\tau_h^i = n) = \frac{1}{\beta\pi(\mathcal{C})}\sum_{j=0}^{\infty}P(\tau = n + j)\nonumber\\
&  \leq M_1\mu(\mathcal{C}) \sum_{j=0}^{\infty} \eta^{n+j} \leq  M_1 \frac{1}{1-\eta}\eta^n.
\end{align}
\end{proof}
\begin{lemma}\label{lemma:geometric}
There exists i.i.d. $k_0-$delayed geometric random variables $\tilde{\tau}$ such that $P(\tilde{\tau}_k \geq n) \geq P(\tau \geq n)$ and $P(\tilde{\tau}_k \geq n ) \geq P(\tau_j^i \geq n).$
\end{lemma}
\begin{proof}
From \eqref{eq:taubounds} it follows that $P(\tau_h^i = n) = \frac{1}{\pi(\mathcal{C})\beta}P(\tau \geq n).$ hence, we have that if we provide an upper bound for the quantity $P(\tau_h^i \geq n)$ we will provide a bound for the quantity $P(\tau \geq n).$ Consequently we will start by working with the random variable $\tau_h^i.$ In particular, we know what $P(\tau_j^i = n) \leq M_2 \eta^n \leq \frac{M_2}{\pi(\mathcal{C})\beta}\eta^n$ Select some $\epsilon > 0$ such that $1-\epsilon > \eta,$ which can always be done. Then, there exists some integer $d$ such that $\frac{M_2}{\beta\pi(\mathcal{C})} \leq (1-\epsilon)^d,$ and therefore
\begin{equation}
P(\tau \geq n) = \frac{P(\tau_h^i)}{\beta\pi(\mathcal{C})} \leq (1-\epsilon)^d \eta^n. 
\end{equation}
Let $\iota \triangleq \frac{\eta}{(1-\epsilon)} < 1,$
then
\begin{equation}
P(\tau \geq n) = \frac{P(\tau_h^i = n)}{\beta\pi(\mathcal{C})} \leq (1-\epsilon)^{n+d}\iota^n
\end{equation}
and for sufficiently large $n,$ $(1-\epsilon)^d\iota^n \leq \epsilon.$ Hence, there exists a sufficiently large $n \geq k_0$ such that for $n \geq k_0$ we have
\begin{equation}
P(\tau \geq n) = \frac{P(\tau_h^i = n)}{\beta\pi(\mathcal{C})} \leq \epsilon(1-\epsilon)^n,
\end{equation} 
which in turn implies that $P(\tau_h^i \geq n ) \leq (1-\epsilon)^{n+1}.$ Then, let $\tilde{\tau}_k$ be a random variable such that
\begin{equation}
\label{eq:delayeddist}
P(\tilde{\tau}_k = n) = \begin{cases}
0 & \text{if } n < k_0, \\
\epsilon(1-\epsilon)^{n-k_0} & \text{otherwise.}
\end{cases}
\end{equation}
Then, $P(\tilde{\tau}_k \geq n) = \frac{(1-\epsilon)^n}{(1-\epsilon)^{k_0}}u(n-k_0) + \mathbbm{1}_{[1,k_0-1]\cap \mathbb{Z}}(n),$ where $u(\cdot)$ denotes the Heaviside step-function and $\mathbbm{1}_A(\cdot)$ denotes the characteristic function of the set $A.$ Consequently, we have that 
\begin{align}
P(\tau_h^i \geq n) \leq P(\tilde{\tau}_k \geq n) \\
P(\tau \geq n) \leq P(\tilde{\tau}_k \geq n).
\end{align}
\end{proof}
\begin{lemma}\label{lemma:cdominate}
For $k$ i.i.d. $k_0-$delayed geometric random variables $\tilde{\tau}_{i},$ $i =1,\hdots,k.$ with distribution \eqref{eq:delayeddist} we have
\begin{align}
&P(\tilde{\tau}_1+\hdots+\tilde{\tau}_k \geq n) \\
& \quad \geq P(\tau_h^i + \tau_1+\hdots+\tau_{k-1} \geq n), \nonumber
\end{align}
and 
\begin{align}
&P(\tilde{\tau}_1+\hdots+\tilde{\tau}_k \geq n) \\
& \quad \geq P(\tau_1+\hdots + \tau_k \geq n).\nonumber
\end{align} 
\end{lemma}
\begin{proof}
We will establish both facts by induction. Note that the statement holds for $k = 1$ directly by using Lemma \ref{lemma:geometric}. Assume now that for some $k$ and arbitrary $n$ it holds that
\begin{align}
& P(\tilde{\tau}_1 + \hdots + \tilde{\tau}_k \geq n) \geq \\
& \quad  P(\tilde{\tau}_h^i + \tau_1 + \hdots + \tau_{k-1} \geq n ) \nonumber\\
&P(\tilde{\tau}_1 + \hdots + \tilde{\tau}_k \geq n) \geq P(\tau_1 + \hdots + \tau_k \geq n). \label{eq:onlyone}
\end{align}
Now we will establish that the induction argument holds for \eqref{eq:onlyone}. The other quantity follows the exact argument so this is left out.

We may write the expansion
\begin{align}
& P(\tau_1 + \hdots + \tau_{k+1} \geq n ) = \\
& \sum_{m} P(\tau_{k+1} = m) P(\tau_1 + \hdots + \tau_k \geq n-m).
\end{align}
since $P(\tau_1 + \hdots + \tau_k \geq n-m) \leq P(\tilde{\tau}_1+\hdots+\tilde{\tau}_k \geq n-m),$ we have that $P(\tau_1 + \hdots + \tau_{k+1} \geq n) \leq P(\tilde{\tau}_{1} + \hdots + \tilde{\tau}_k + \tau_{k+1} \geq n).$ Further since $P(\tilde{\tau}_1 + \hdots + \tilde{\tau}_k + \tau_{k+1} \geq n) = \sum_{m} P(\tau_{k+1} \geq n-m)P(\tilde{\tau}_1 + \hdots + \tilde{\tau}_k = m) $ and $P(\tau_{k+1} \geq n-m) \leq P(\tilde{\tau}_{k+1} \geq n-m)$ from Lemma \ref{lemma:geometric} we conclude that $P(\tau_1 + \hdots + \tau_{k+1} \geq n) \leq P(\tilde{\tau}_1 + \hdots + \tilde{\tau}_{k+1} \geq n)$ for any $k \geq 1$ and $n \geq 1.$
\end{proof}

\begin{lemma}\label{lemma:delay} Let $\tilde{\tau}_1,\hdots,\tilde{\tau}_k$ denote $k$ $k_0-$delayed i.i.d. geometric distributions with probability of success $\epsilon.$ Further, let $\xi_1,\hdots, \xi_k$ denote $k$ i.i.d. geometric random variables with probability of success $\epsilon.$ Then, 
\begin{equation}
P(\tilde{\tau}_1 + \hdots + \tilde{\tau}_k \geq n) = P(\xi_1 + \hdots + \xi_k \geq n - k(k_0-1))
\end{equation}
\end{lemma}

\begin{proof}
We will first establish that the equality 
\begin{equation}
P(\tilde{\tau}_1 + \hdots + \tilde{\tau}_k = n) = P(\xi_1 + \hdots + \xi_k = n - k(k_0-1))
\end{equation}
holds. Then, in order to claim that the lemma holds true, all we have to do is sum over $n.$ Since we are dealing with the sum of independent random variables, the distribution of the sum will be the convolution of distributions. We will be using the Z-transforms as then we can conveniently deal with products instead. Let $U_{\tilde{\tau}}(z)$ denote the $Z-$transform of $P(\tilde{\tau} = n),$ i.e.,
\begin{equation}
U_{\tilde{\tau}}(z) = \epsilon \frac{z^{-(k_0-1)}}{z-\epsilon}.
\end{equation} 
Analogously,
\begin{equation}
U_{\xi}(z) = \epsilon \frac{1}{z-\epsilon}.
\end{equation}
Then, the Z-transform of each of the $k-$fold convolutions can be written as
\begin{align}
&U_{\tilde{\tau}}^k(z) = \epsilon^k \frac{z^{-k(k_0 -1)}}{(z-\epsilon)^k} \\
& U_{\xi}^k(z) = \frac{\epsilon^k}{(z-\epsilon)^k}.
\end{align}
It then follows that $P(\tilde{\tau}_1 + \hdots + \tilde{\tau}_k = n) = P(\xi_1 + \hdots + \xi_k = n-k(k_0-1)).$
\end{proof}
Let us define $N_{\mathcal{C}_1}(n-j,i)$ to be the number of times the process visits the set $\mathcal{C}$ in time $n-j$ after time $i.$ The quantity is directly dependent on the inter-arrival times $\tau_h^i,\tau_1,\hdots.$
We now introduce a final supporting lemma regarding the quantity $N_{\mathcal{C}_1}(n-j,i)$
\begin{lemma}\label{lemma:numbervisits}
The number of renewals $N_{\mathcal{C}_1}(n-j,i)$ associated to the inter-renewal times $\tau_h^i,\tau_1,\hdots$ after time $i$ in time $n-j$ and the number of renewals $\tilde{N}(n-j)$ associated by the inter-arrival times $\tilde{\tau}_1,\tilde{\tau}_2,\hdots$ fulfil the following relationship
\begin{equation}
P(N_{\mathcal{C}_1}(n-j,i) \leq k) \leq P(\tilde{N}(n-j)\leq k),\,\forall k.
\end{equation}
\end{lemma}
\begin{proof}
Since
\begin{align}
P(N_{\mathcal{C}_1}(n-j,i)\leq k) = P(\tau_h^i + \tau_1 + \tau_{k-1} \geq n-j) \\
P(\tilde{N}(n,j) \leq k) = P(\tilde{\tau}_1 + \hdots + \tilde{\tau}_k \geq n-j)
\end{align}
from lemma \ref{lemma:cdominate} it follows that
\begin{equation}
P(N_{\mathcal{C}_1}(n-j,i) \leq k) \leq P(\tilde{N}(n-j)\leq k),\,\forall k.
\end{equation} 
\end{proof}
We are now have all the supporting Lemmas required to establish Theorem \ref{theorem:contraction}. Note that for any $i \leq n$ we have
\begin{equation}
\mathbb{E}\left[ \Pi_{i=j}^n q(\boldsymbol{\Theta}_i)\right] \leq \mathbb{E}\left[ q_{\mathcal{C}}^{N_{\mathcal{C}_1}(n-j,i)}\right],
\end{equation}
where $q_{\mathcal{C}}$ denotes the worst contraction parameter within the set $\mathcal{C}$. Then, from Lemma \ref{lemma:numbervisits} it follows that
\begin{equation}
P(q_{\mathcal{C}}^{N_{\mathcal{C}_1}(n-j,i)} \geq q_{\mathcal{C}}^k) \leq P(q_{\mathcal{C}}^{\tilde{N}(n-j)} \geq q_{\mathcal{C}}^k). 
\end{equation}
The above implies that
\begin{equation}
\mathbb{E}[q_{\mathcal{C}_1}^{N_{\mathcal{C}}(n-j)}] \leq \mathbb{E}[q_{\mathcal{C}}^{\tilde{N}(n-j)}].
\end{equation}
Consequently, we will now work with the quantity $\mathbb{E}[q_{\mathcal{C}}^{\tilde{N}(n-j)}].$
The expectation can be explicitly written as 
\begin{equation}
\mathbb{E}[q_{\mathcal{C}}^{\tilde{N}(n-j)}] = \sum_{k=0}^{\lfloor\frac{n-j}{k_0} \rfloor} q_{\mathcal{C}}^k P(\tilde{N}(n-j)=k),
\end{equation}
where the last term of the sum corresponds to $\lfloor \frac{n-j}{k_0} \rfloor.$ This is due to the fact that in time $n-j$ a maximum of $\lfloor \frac{n-j}{k_0}\rfloor$ visits to $\mathcal{C}$ are possible due to the variables $\tilde{\tau}_i$ being $k_0$ delayed. Further, using Lemma \ref{lemma:delay} we have that
\begin{align}
& \mathbb{E}[q_{\mathcal{C}}^{\tilde{N}(n-j)}] = \\
& \sum_{k=0}^{\lfloor \frac{n-j}{k_0}\rfloor}q_{\mathcal{C}}^kP(\xi_i+\hdots+\xi_k \geq n-j-k(k_0-1)). \nonumber
\end{align}
Using Chernoff's bound we have that
\begin{align}
P(\xi_1 + \hdots + \xi_k \geq n-j-k(k_0-1)) \leq \\
\mathbb{E}[e^{t\xi}]^k e^{-t(n-j-k(k_0-1))}
\end{align} for any $t > 0.$ Further, since $\xi$ is a geometric random variable with success probability $\epsilon >0$ we have
\begin{equation}
\mathbb{E}[e^{t\xi}]^k = \left(\frac{\epsilon e^t}{1-(1-\epsilon)e^t} \right)^k.
\end{equation}
For notational simplicity let $m(t) \triangleq \mathbb{E}[e^{t\xi}].$ Note that $m(t)$ can be made arbitrarily close to 1 by appropriate selection of $t.$ Then,
\begin{align}
& \mathbb{E}[q_{\mathcal{C}}^{\tilde{N}(n-j)}] \leq \\
& e^{-t(n-j)}\sum_{k=0}^{\lfloor \frac{n-j}{k_0} \rfloor}
\left(q_{\mathcal{C}}m(t)e^{t(k_0-1)} \right)^k = \nonumber\\
& e^{-t(n-j)}\frac{1-\left(q_{\mathcal{C}}m(t)e^{t(k_0-1)} \right)^{\lfloor \frac{n-j}{k_0}\rfloor}+1}{1-q_{\mathcal{C}}m(t)e^{t(k_0-1)}}.
\end{align}
If $q_{\mathcal{C}}m(t)e^{t(k_0-1)} < 1$ the exponential decay in $n-j$ can be seen immediately. Writing out the condition by using the definition of $m(t)$ we have
\begin{equation}
\label{eq:condition}
q_{\mathcal{C}}\left(\frac{\epsilon e^{tk_0}}{1-(1-\epsilon)e^t} \right) < 1.
\end{equation}
Note that $m(t)e^{t(k_0-1)} > 1$ and equality holds if $t = 0.$ Hence, since \eqref{eq:condition} is monotonically increasing and continuous for $t \in [0,\text{ln}(1/(1-\epsilon))],$ $t$ can be chosen arbitrarily close to 0, yielding a value of  $m(t)e^{t(k_0-1)}$ that can be made arbitrarily close to 1. Since $q_{\mathcal{C}} < 1,$ there exists a sufficiently small value of $t$ that makes \eqref{eq:condition} true.
Therefore, there exists $C < \infty$ and $\gamma < 1$ such that
\begin{equation}
\mathbb{E}[\Pi_{j=i}^n q(\boldsymbol{\Theta}_i)] \leq C \gamma^{n-j}.
\end{equation}
\ifCLASSOPTIONcaptionsoff
  \newpage
\fi

\end{document}